\documentclass[submission]{FPSAC2022_arxiv}

\newtheorem{thm}{Theorem}

\newtheorem{prop}[thm]{Proposition}
\newtheorem{remark}[thm]{Remark}
\newenvironment{rem}{\begin{remark}\normalfont}{\end{remark}}

\usepackage{lipsum}

\title[Peaks of cylindric plane partitions]{Peaks of cylindric plane partitions}

\author[D. Betea, A. Occelli]{Dan Betea\thanks{\href{mailto:dan.betea@gmail.com}{dan.betea@gmail.com}; supported by FWO Flanders project EOS 30889451.}\addressmark{1} \and Alessandra Occelli\thanks{\href{mailto:alessandra.occelli@tecnico.ulisboa.pt}{alessandra.occelli@tecnico.ulisboa.pt}; supported by the HyLEF ERC starting grant 2016. This project was completed while A.O. was a postdoctoral fellow at MSRI during the Program ``Universality and Integrability in Random Matrix Theory and Interacting Particle Systems''.}\addressmark{2,}\addressmark{3}}

\address{\addressmark{1} Department of Mathematics, KU Leuven, Belgium \\ \addressmark{2} Department of Mathematics, Instituto Superior T\'ecnico, Lisbon, Portugal \\ \addressmark{3} Mathematical Sciences Research Institute, Berkeley, CA, USA}

\received{\today}

\abstract{We study the asymptotic distribution, as the volume parameter goes to 1, of the peak (largest part) of finite- or slowly-growing-width cylindric plane partitions weighted by their trace, seam, and volume. There are two natural asymptotic regimes depending on the trace/seam parameters, and in both cases we obtain asymptotics governed by finite temperature (periodic) analogues of the Bessel and Airy gap probabilities from random matrix theory. In particular, the distributions we obtain interpolate \emph{in more than one way} between two well-known extremal value distributions: the Gumbel distribution of maxima of iid random variables and the Tracy--Widom distribution of maxima of eigenvalues of random Hermitian matrices. We also interpret our results in terms of last passage percolation on a cylinder, which yields to interesting connections to the Kardar--Parisi--Zhang equation.}

\usepackage[backend=bibtex]{biblatex}
\usepackage{amsmath, amsthm, amssymb, verbatim, amsfonts, bbm, mathtools}
\usepackage{color}

\def\Z{\mathbb{Z}}
\def\N{\mathbb{N}}
\def\R{\mathbb{R}}
\def\P{\mathbb{P}}

\def\Id{\mathbbm{1}}
\def\Li{\mathrm{Li}_2}
\def\tr{\mathrm{tr}\,}
\def\sm{\mathrm{sm}\,}
\def\Bessel{\mathrm{\,Bessel}}
\def\Airy{\mathrm{\,Airy}}
\def\Ai{\mathrm{\,Ai}}

\begin{document}

\maketitle

\section{Introduction and main results}

\paragraph{Background and motivation.} Cylindric plane partitions are better explained in a picture than a formula: they are plane partitions wrapped around the cylinder as in Figure~\ref{fig:cyl} (left). They have been studied in the combinatorial and probabilistic literature in various guises, and we mention the works of Gessel--Krattenthaler~\cite{GK97} and Borodin~\cite{Bor07} as two notable examples. We note that the latter article is probabilistic in nature and deals with the local nature of the correlation functions ``in the bulk'' as the size of the partitions (or the volume parameter) grows (resp.\,approaches 1). 

It has been known for a while that cylindric boundary conditions for combinatorial objects correspond to finite temperature systems in the sense of quantum mechanics, with inverse temperature the same as the cylinder's circumference. An example to illustrate both the combinatorics and the mathematical physics is the cylindric Plancherel measure discussed in~\cite{BB19} (and introduced in~\cite{Bor07}). The motivation for this paper stems from there: by studying simple distributions on cylindric plane partitions, we obtain interesting asymptotic behavior for their peaks that ``crosses over'' between a large (infinite) temperature regime governed by the Gumbel distribution (universally the asymptotic maximum of iid random variables that have Gaussian-like tails) and a small (zero) temperature regime governed by the Tracy--Widom~\cite{TW94_airy} distribution (universally the asymptotic maximum of correlated spectra of Hermitian matrices with iid entries).

\paragraph{Main result.} In this work we initiate the study of the peaks/largest parts of such partitions. We limit ourselves to a cylinder of finite or slowly growing circumference $2N$ and consider the case of volume-trace-and-seam-distributed such partitions that contain only one peak (and also only one ``bottom''). In technical terms, we look at cylindric plane partitions $\Lambda$ that are in bijection with a sequence of $2N$ interlacing ordinary integer partitions as follows:
\begin{equation} \label{eq:interlacing}
    \Lambda = \{ \mu = \lambda^{(-N)} \prec \cdots \prec \lambda^{(-1)} \prec \lambda^{(0)} = \lambda \succ \lambda^{(1)} \succ \cdots \succ \lambda^{(N)} = \mu \}. 
\end{equation}
The largest (ordinary) partition $\lambda$ gives the peak $\lambda_1$ of $\Lambda$ (the size of the inner-most central column of boxes), and also the trace: $\tr \Lambda := |\lambda|$ where $|\lambda| = \sum_i \lambda_i$ denotes the size of an integer partition. The smallest partition $\mu$ gives the seam of $\Lambda$: $\sm \Lambda := |\mu|$. The volume of $\Lambda$, denoted $|\Lambda|$, is the number of cubes in it: $|\Lambda| := \sum_{-N \leq i < N-1} |\lambda^{(i)}|$.


We fix a positive integer $N$ and two numbers $0 \leq a, q \leq 1$ such that $aq < 1$. The probability distribution considered throughout is the following trace-seam-volume one:
\begin{equation} \label{eq:cyl_dist}
    \P(\Lambda) = \frac{a^{\tr \Lambda} \cdot (a^{-1} q^{-N})^{\sm \Lambda} \cdot q^{|\Lambda|}} {Z} = \frac{a^{|\lambda|} \cdot (a^{-1} q^{-N})^{|\mu|} \cdot q^{\sum_i |\lambda^{(i)}|}} {Z}
\end{equation}
where $Z = (q^{-N}; q^{-N})^{-1}_\infty \prod_{1 \leq i, j \leq N} (q^{i+j-1}; q^N)_\infty^{-1}$ is the total weight/partition function (and $(x; q)_n = \prod_{0 \leq i \leq n-1} (1-x q^i)$ is the $q$-Pochhammer symbol). Our main result concerns the asymptotic distribution of $\lambda_1$ as $q \to 1-$. It is as follows.

\begin{thm} \label{thm:main}
    Let $\lambda_1$, using~\eqref{eq:interlacing}, be the largest part of a width $2N$ cylindric plane partition distributed as in~\eqref{eq:cyl_dist}. Then 
    \begin{itemize}
        \item (i) ($N$ fixed, $a$ grows critically) if $q = e^{-\epsilon}, a = e^{-\alpha \epsilon}$ as $\epsilon \to 0+$ and for $N$ fixed we have:
        \begin{equation} \label{eq:ftB_thm}
            \lim_{\epsilon \to 0+} \P \left( \frac{\lambda_1 - 2 \epsilon^{-1} \log \epsilon^{-1} } {\epsilon^{-1} } < s \right) = F^{(N)}_{\alpha, \Bessel} = \det(1-K^{(N)}_{\alpha, \Bessel})_{L^2(s, \infty)}
        \end{equation}
        where the determinant on the right is a Fredholm determinant of the following finite temperature Bessel kernel (operator) in exponential coordinates
        \begin{equation} \label{eq:ftB_kernel}
            K^{(N)}_{\alpha, \Bessel} (x,y) = e^{-\frac{x}{2}-\frac{y}{2}} \int_{0}^{\infty} \frac{1}{1+u^N} J_\alpha (2 \sqrt{u e^{-x}}) J_\alpha (2 \sqrt{u e^{-y}}) du
        \end{equation}
        with $J$ the Bessel function of the first kind;
        \item (ii) ($N$ grows slowly, $a$ fixed) if $q = e^{-\epsilon}, N = \beta \epsilon^{-2/3}$ as $\epsilon \to 0+$ and with $0 < a < 1$ fixed, we have (below $c_1 = - 2 \log (1-\sqrt{a}), c_2 = 2^{-1/3} a^{1/6} (1-\sqrt{a})^{-2/3}$):
        \begin{equation} \label{eq:ftA_thm}
            \lim_{\epsilon \to 0+} \P \left( \frac{\lambda_1 - c_1 \epsilon^{-1} } {c_2 \epsilon^{-1/3} } < s \right) = F^{(\beta c_2)}_{\Airy} = \det(1-K^{(\beta c_2)}_{\Airy})_{L^2(s, \infty)}
        \end{equation}    
        where the Fredholm determinant on the right is for the finite temperature Airy kernel
        \begin{equation} \label{eq:ftA_kernel}
            K^{(\beta)}_{\Airy} (x,y) = \int_{-\infty}^{\infty} \frac{e^{\beta v}}{1 + e^{\beta v}} \Ai(x+v) \Ai(y+v) dv
        \end{equation}
        with $\Ai$ the (exponentially decaying at $\infty$) Airy function.
    \end{itemize}
\end{thm}

The distributions appearing above are Fredholm determinants. Recall that an operator $K$ on $L^2(X)$ and having kernel $K(x,y)$ ($X=(s, \infty)$ for us) acts via matrix multiplication $(Kf) (x) = \int_X K(x, y) f(y) dy$. If $K$ is trace class---see e.g.~\cite{Rom15}, the \emph{Fredholm determinant} of $1-K$ ($1$ the identity operator) on $L^2 (X)$ is the quantity
\begin{equation}
    \det(1-K)_{L^2(X)} = 1 + \sum_{m \geq 1} \frac{(-1)^m}{m!} \int_X \cdots \int_X \det_{1 \leq i, j \leq m} [K(x_i, x_j)] d x_1 \cdots d x_m.
\end{equation}

\paragraph{A few important remarks.} We make the following observations.

\begin{rem}
The kernel $K^{(N)}_{\alpha, \Bessel}$ is similar to the finite-temperature Bessel kernel recently introduced in the physics literature by Lacroix-A-Chez-Toine et al~\cite{LDMS18}, at inverse finite temperature $N$. 
\begin{itemize}
    \item As $N \to \infty$ the Fermi factor $1/(1+u^N)$ becomes the indicator $\Id_{[0,1]}$ and the kernel becomes
    \begin{equation}
        K^{(\infty)}_{\alpha, \Bessel} (x, y) = e^{-\frac{x}{2}-\frac{y}{2}} B_{\alpha} (e^{-x}, e^{-y})
    \end{equation}
    with $B_{\alpha}(x, y) = \int_{0}^{1} J_\alpha (2 \sqrt{u x}) J_\alpha (2 \sqrt{u y}) du$ the Bessel kernel found in the scaling of the smallest eigenvalue of random matrix ensembles with a hard edge (e.g.\, the Laguerre or Jacobi ensembles scaled around 0, see~\cite{TW94_bessel}). Furthermore, the probability distribution $F^{(\infty)}_\alpha (s) = \det(1 - K^{(\infty)}_{\alpha, \Bessel})_{L^2(s, \infty)}$ (already at ``$0 = 1/(N=\infty)$ temperature'') interpolates between the standard Gumbel distribution $(\alpha = 0)$ and the Tracy--Widom distribution ($\alpha \to \infty$)~\cite{Joh08}.
    \item As $N \to 0$ (a combinatorially meaningless limit!) we have (upon a change of variables and using the orthogonality of Bessel $J$ functions) 
    \begin{equation}
        K^{(0)}_{\alpha, \Bessel} (x, y) = e^{-\frac{x}{2}-\frac{y}{2}} \delta(x - y)
    \end{equation}
    (in a distributional sense, with $\delta$ the Dirac delta measure) and $F^{(0)}(s) = \det(1 - K^{(0)}_{\alpha, \Bessel})_{L^2(s, \infty)} = e^{-e^{-s}}$ becomes the Gumbel distribution.
\end{itemize}
\end{rem}

\begin{rem}
    The kernel $K^{(\beta c_2)}_{\Ai}$ is Johansson's~\cite{Joh07} finite temperature Airy kernel (at inverse temperature $\beta c_2$). When $\beta \to \infty$ it becomes the usual Airy kernel (by the same argument as above, the Fermi factor in the integrand becoming $\Id_{[0,\infty)}$) with Fredholm determinant on $(s, \infty)$ the Tracy--Widom GUE distribution~\cite{TW94_airy} (the ``universal'' asymptotic distribution of the largest eigenvalue of Hermitian random matrices with iid entries). A more subtle $\beta \to 0$ limit recovers the Gumbel distribution---see~\cite{Joh07, BB19}. 
\end{rem}

\begin{rem}
    As already observed in~\cite{LDMS18}, one can presumably transition from our finite temperature Bessel kernel to the finite temperature Airy kernel by using the well-known asymptotics for Bessel functions~\cite[eq.~10.19.8]{DLMF}:
    \begin{equation}
        (2/\nu)^{\tfrac13} J_\nu (\nu + x \nu^{\tfrac13}) = \Ai(-2^{\tfrac13} x) + O(\nu^{-\tfrac23}), \quad \nu \to \infty 
    \end{equation}
    but we shall not pursue this purely analytical point here. This gives rise to the hard-to-soft edge transition in random matrix theory, see e.g.\,\cite{BF03}.
\end{rem}

\begin{rem}
    There are other limits one can take not covered here. For example, if $q \to 1$ while $a$ and $N$ are fixed, we expect Tracy--Widom fluctuations (zero temperature behavior). We also expect Tracy--Widom if $0<q<1$ is fixed (even if $a \to 1$ and possibly even if $N$ is growing). More interestingly, we can take $q, a \to 1$ as $N \to \infty$. Here the behavior is unexplored and we hope to address this in future work. There are interesting limits as $q$ and/or $a$ go to 0, and one of them leads to the cylindric Plancherel measure of~\cite{Bor07} and~\cite[Thm.~1.1]{BB19}.
\end{rem}

\begin{rem}
    One limit that is worth mentioning as it is combinatorial is that of $N \to \infty$ independent of $a, q$. If we take $N \to \infty$ with $a, q$ fixed then the cylindric geometry goes away (the sides go to infinity, the seam becomes $0$) and we look at regular trace-and-volume-weighted plane partitions. These were covered in~\cite{BO21} (see in particular Remark 6 there). Asymptotically we get either Bessel (in exponential coordinates) behavior or Tracy--Widom behavior depending on whether $a \to 1$ or is fixed. Alternatively, we can take $N$ to infinity after taking $a, q \to 1$ (see above), and we recover the same behavior. The two limits commute. 
\end{rem}

\begin{rem}
    An extended version of Thm.~\ref{thm:main} is also possible. That it, if one looks at the top $k$ lozenges in a neighborhood of the peak, the joint distribution of those is also asymptotically a Fredholm determinant, this time of a ``time-extended'' kernel. In the Bessel limit perhaps this is not too interesting as ``time'' (partition number around 0 in~\eqref{eq:interlacing}) remains discrete. In the continuous case though a rescaling is needed and one would then obtain, as was done in~\cite{BB19}, the finite temperature/periodic Airy process of~\cite{DMS17}.
\end{rem}

\begin{rem}
    Let us notice one last thing: for certain parameter ranges, our main theorem combined with recent results of~\cite{ARP21} could perhaps elucidate certain (perhaps expected) behavior of maxima of Gaussian free fields on a cylinder.
\end{rem}

\begin{figure} [ht!]
    \begin{center}
        \includegraphics[scale=0.5]{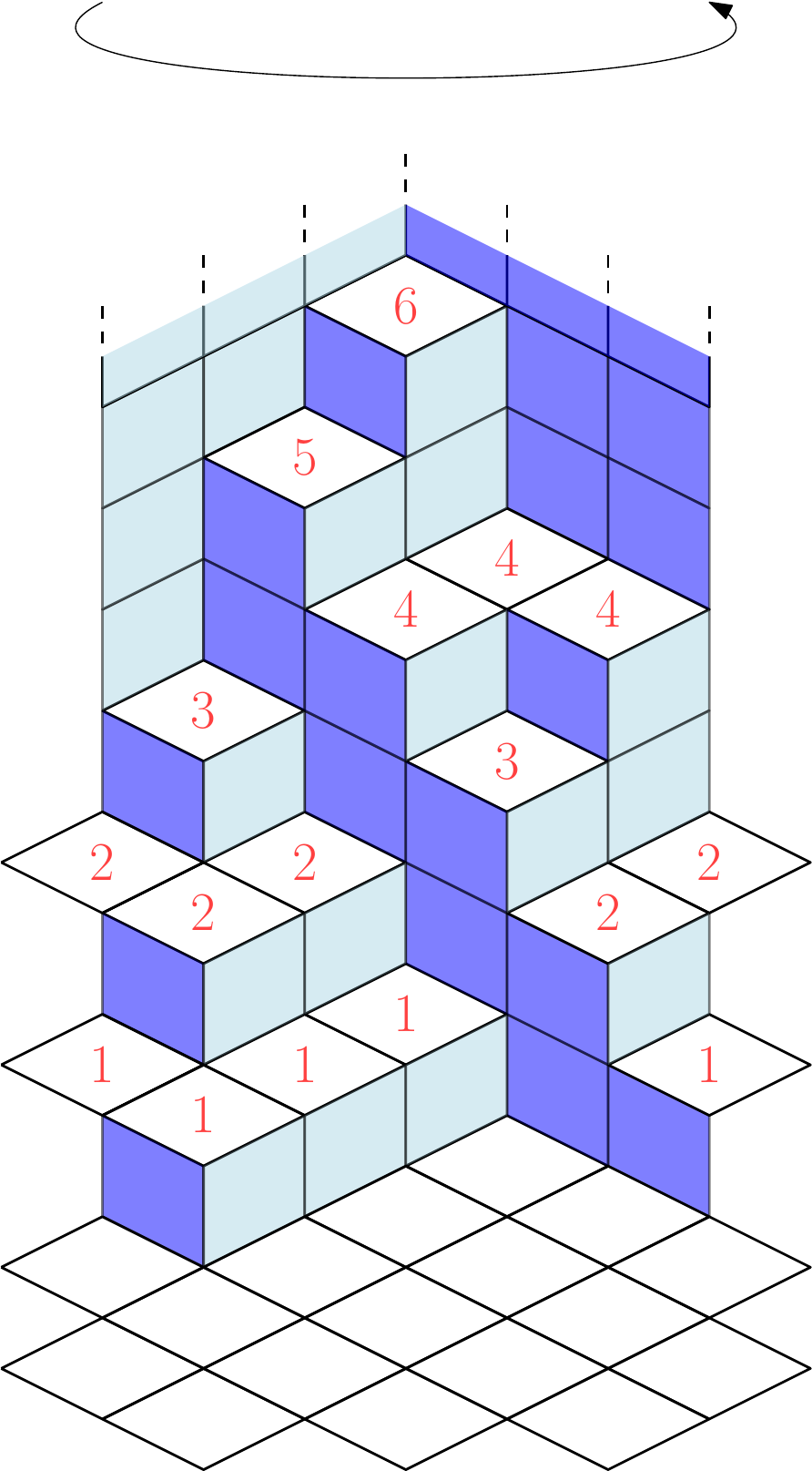}  \qquad  \includegraphics[scale=0.4]{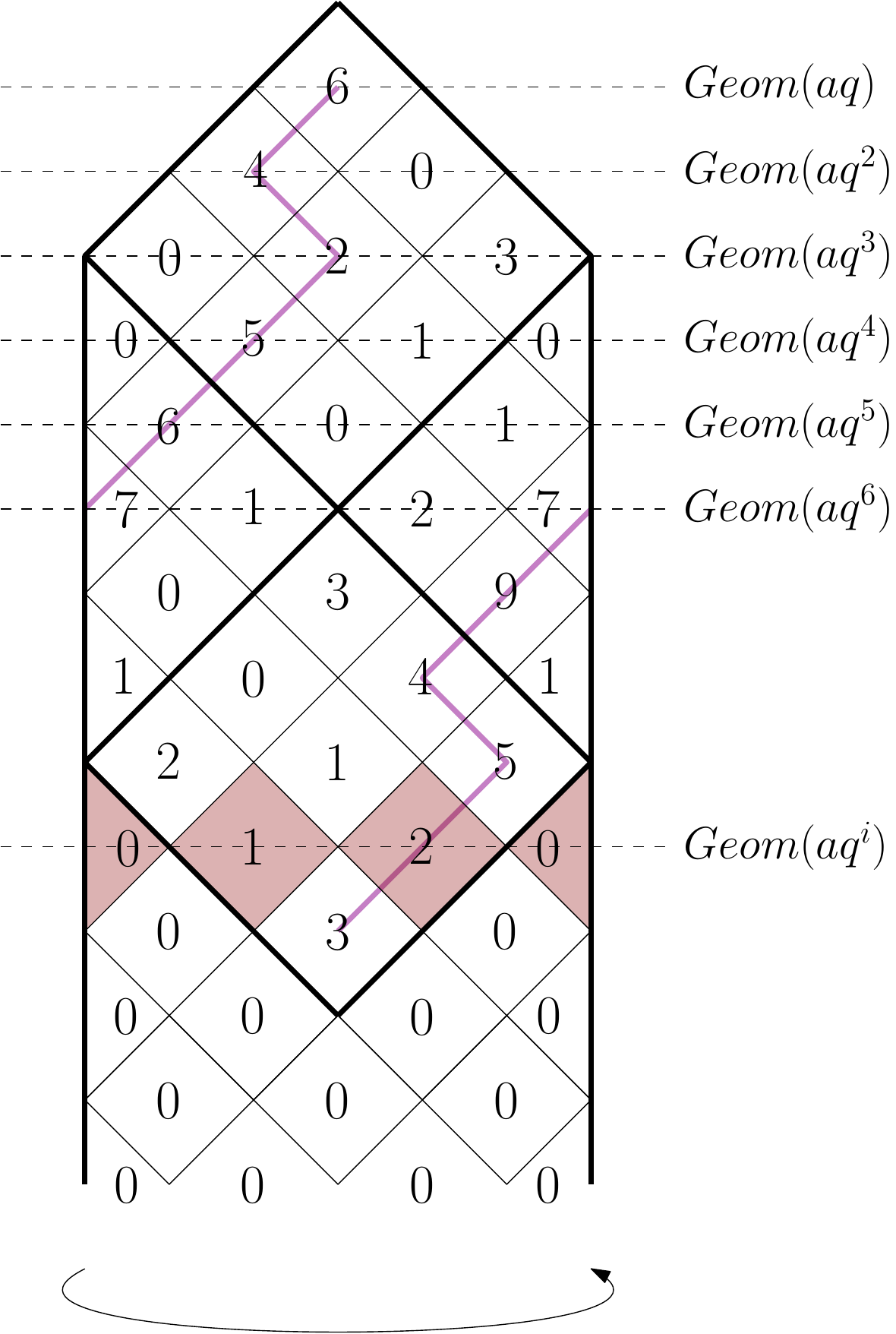}
    \end{center}
    \caption{\footnotesize{Left: a cylindric plane partition of width $2N = 6$, seam $3$, trace $11$, volume $44$, and a peak of $6$. Right: a depiction of last passage percolation on a tie. Each horizontal row $i$ of lozenges has iid $Geom(a q^i)$ random variables inside each lozenge. The longest down-right/left path has length $L=53$.}}
    \label{fig:cyl}
\end{figure}

\section{Last passage percolation and $q$-Whittaker measures}

In this section we connect the above results with a model of directed last passage percolation on a cylinder. Similar models were studied in~\cite{BBNV_FPSAC_19} for the case of reflecting strips (equivalently, vertically symmetric cylinders), and while the combinatorial tools apply mutatis-mutandis, the probability measures studied are different.

Fix numbers $0 < a, q < 1$ (for simplicity). Let us consider the following so-called last passage percolation on a tie (see also~\cite{BBNV_FPSAC_19}). In Figure~\ref{fig:cyl} (right) the two vertical sides are identified, so we are on a cylinder. The cylinder, of width $2N$, is tiled with $N \times N$ big squares (with bolded edges, not so important for what follows), and each such big square contains $N^2$ unit squares. In each unit square we have a geometric random variable, and they are all independent of one another. On horizontal slice $i \geq 1$, where a slice is a sequence of squares touching only at the corners and sharing a common horizontal line through their centers, the random variables are iid ${\rm Geom} (a q^i)$. (A random variable $X$ is ${\rm Geom} (t)$ if $\P(X=k) = (1-t) t^k$. Note that slice $i$ contains $\min (i, N)$ squares.) 

In the geometry described above, consider paths going from the top-most unit square down to $\infty$ and using only unit down-right or down-left steps. The length of a path is the sum of the random geometric integers it encounters. Now consider the longest such path, and call the length $L$. By Borel--Cantelli, the cylinder contains only a finite number of squares whose random variables are non-zero, and thus $L < \infty$ almost surely. We then have the following theorem. 

\begin{thm} \label{thm:lpp}
    It holds that 
    \begin{equation}
        L + \chi = \lambda_1 \quad \text{in\ distribution}
    \end{equation}where $\lambda_1$ is the peak of a cylindric plane partition distributed according to~\eqref{eq:cyl_dist}, $L$ is as above, and $\chi$ is a random variable independent of $L$ with explicit distribution $\P(\chi \leq m) = (q^{N(m+1)}; q^N)_\infty$. (In fact, $\chi = \kappa_1$ is the distribution of the first part of a random integer partition $\kappa$ with $\P(\kappa) = (q^N; q^N)^{-1}_\infty (q^N)^{|\kappa|}$).
\end{thm}

\begin{proof}
    The proof is the same as that of Thm.~2.1 of~\cite{BBNV_FPSAC_19}, with the exception that there the sides of the tie were reflecting, and here they are identified. In other words free boundary conditions are replaced by periodic boundary conditions. Note the proof uses nothing more than the classical Robinson--Schensted--Knuth correspondence~\cite{Knu70} and Greene's theorem~\cite{Gre74}, adapted to periodic/free boundary conditions. 
\end{proof}

Let us state another important remark. In recent work of Imamura--Mucciconi--Sasamoto~\cite{IMS21} it was proven that $L$ (via $\lambda_1$) is related to an observable of a so-called $q$-Whittaker measure. The $q$-Whittaker symmetric polynomials $P_\rho(x; q)$ and $Q_\rho(x; q)$ are the Macdonald $(q,t)$ symmetric polynomials~\cite{Mac95} in variables $x = (x_1, \dots, x_N)$ at $t=0$. What was proven in~\cite{IMS21} together with the above yields the following result. 

\begin{prop}
    One has $L + \chi = \rho_1 + \chi'$ in distribution where $\rho_1$ is the first part of a random partition distributed with the $q$-Whittaker measure
    \begin{equation}
        \P(\rho) = \tilde{Z}^{-1} P_\rho (\sqrt{a} q^{\frac{1}{2}}, \dots, \sqrt{a} q^{N - \frac{1}{2}}; q^N) Q_\rho (\sqrt{a} q^{\frac{1}{2}}, \dots, \sqrt{a} q^{N-\frac{1}{2}}; q^N)
    \end{equation}
    with $\tilde Z = \prod_{1 \leq i, j \leq N} (q^{i+j-1}; q^N)_\infty^{-1}$ (each polynomial has $N$ variables $x_i = \sqrt{a} q^{i-\tfrac12}, 1 \leq i \leq N$), and $\chi, \chi'$ are independent of everything with the same distribution as in Thm.~\ref{thm:lpp} $\P(\chi \leq m) = (q^{N(m+1)}; q^N)_\infty$.
\end{prop}

The importance of such measures was first realized by Borodin--Corwin~\cite{BC14} who showed that $\rho_1$ (for $\rho$ $q$-Whittaker distributed for certain specializations of symmetric functions \emph{different from the ones above}) converges to the height function of the so-called stochastic heat equation (a relative of the Kardar--Parisi--Zhang or KPZ equation). 

Our main asymptotic result of this paper together with the discussion in this section seem to then suggest that one can also see the hard edge of random matrix ensembles in the long-time behavior of KPZ models. Notice that the finite temperature Airy kernel is already known to appear in the solution of the KPZ equation at finite time with wedge initial conditions. See for example~\cite{ACQ11}. We plan to investigate this further in future work.

\section{Proof of Theorem~\ref{thm:main}}

We now attempt to sketch a proof of Theorem~\ref{thm:main}. We start from the measure~\eqref{eq:cyl_dist}. Via the representation of~\eqref{eq:interlacing}, it is not hard to see that the measure can be rewritten as
\begin{equation}
    \P(\Lambda) = Z^{-1} (q^{N})^{|\mu|}  \prod_{-N \leq i \leq -1} s_{ \lambda^{(i+1)} / \lambda^{(i)}} (\sqrt{a} q^{|i|-1/2}) \prod_{1 \leq i \leq N} s_{ \lambda^{(i-1)} / \lambda^{(i)}} (\sqrt{a} q^{i-1/2}) 
\end{equation}
where we recall $\lambda^{(\pm N)} = \mu, \lambda^{(0)} = \lambda$ by definition. Here $s_{\lambda / \mu}$ are the skew Schur polynomials~\cite[Ch.~1]{Mac95}, and they are evaluated in one variable, so the above is immediate as $s_{\lambda / \mu} (x_1) = x_1^{|\lambda| - |\mu|} \Id_{\mu \prec \lambda}$. By the branching rule we have that the $\lambda$ marginal is
\begin{equation}
    \P(\lambda) \propto \sum_{\mu} q^{N |\mu|} \left[ s_{\lambda/\mu} (\sqrt{a} q^{\tfrac12}, \dots, \sqrt{a} q^{N-\tfrac12})\right]^2
\end{equation} 
and it turns out such measures, called cylindric or periodic Schur measures, were studied before by Borodin (who also derived bulk asymptotics for periodic plane partitions) and Betea--Bouttier (who derived edge asymptotics for a simpler periodic Plancherel measure). 

The point process associated to the partition $\lambda$, that is the set $\{\lambda_i - i + 1/2\}$, is nonetheless not determinantal (despite a certain determinantal structure already present above). However, the \emph{shift-mixed or grand canonical} point process $\{ \lambda_i - i + 1/2 + c \}$, where $c$ is a random variable independent of everything else with explicit distribution $\P(c) = t^c q^{N c^2/2} / \theta_3(t; q^N)$ is indeed determinantal, with $t$ a parameter which we can fix to be 1 throughout\footnote{This shift is sometimes called ``passage to the grand canonical ensemble'' in the physics literature.}. (Here $\theta_3 (t; u) := \sum_{c \in \Z} t^c u^{c^2/2} = (u; u)_\infty (-\sqrt{u} z; u)_\infty (-\sqrt{u}/z; u)_\infty$ is one of the Jacobi theta functions.) Precisely it means that if $\{ k_1, \dots, k_n \} \in \Z+\tfrac12$ is a finite set of positions, the following determinantal identity holds~\cite{Bor07, BB19}:
\begin{equation}
    \P(\{ k_1, \dots, k_n \} \in \{ \lambda_i - i + 1/2 + c \}) = \det_{1 \leq i,j \leq n} K(k_i, k_j)
\end{equation}
where $K$ is the following discrete kernel/operator (below we follow notation from~\cite{BB19}):
\begin{equation}
    K(k, k') = \frac{1}{(2 \pi i)^2} \int_{|z| = 1+\delta} \frac{dz}{z^{k+1}} \int_{|w| = 1-\delta} \frac{dw}{w^{-k'+1}} \cdot \frac{F(z)}{F(w)} \cdot \kappa(z, w)
\end{equation}
where $\delta$ is a small positive number and 
\begin{equation}
    F(z) := \frac{(\sqrt{a} q^{1/2}/z; q)_\infty}{(\sqrt{a} q^{1/2} z; q)_\infty}, \quad \kappa(z, w) := \sqrt{\frac{w}{z}} \cdot \frac{ (q^N; q^N)^2_\infty } { \theta_{q^N} (w/z) } \cdot \frac{\theta_3(t \frac{z}{w}; q^N)} {\theta_3(t; q^N)}. 
\end{equation}
(Here $\theta_u(x) := (x; u)_\infty (u/x; u)_\infty$ is another theta function, in multiplicative notation.) The contour integrals above signify coefficient extraction, and make sense analytically (this becomes important for asymptotic analysis!) as long as the $z$ contour is outside the $w$ contour and both are very close to the unit circle\footnote{A sufficient set of conditions has i) $1 < \frac{|z|}{|w|} = \frac{1+\delta}{1-\delta} < q^{-N}$; ii) $1+\delta = |z| < (aq)^{-1/2}$ and iii) $1-\delta = |w| > (aq)^{1/2}$---see~\cite{BB19}.}.  

The principle of inclusion-exclusion then gives that the distribution of $L = \lambda_1 - 1/2 + c$ is a discrete Fredholm determinant of $K$: $\P(L \leq \ell) = \det(1-K)_{l^2 \{\ell+1/2,\ell+3/2,\dots\}}$. 

We now show, in the scaling regime of part (i), that $\det(1-K) \to \det(1-K^{(N)}_{\alpha, \Bessel})$ in the limit $\ell = \epsilon^{-1}s + 2 \epsilon^{-1} \log \epsilon^{-1}$ as $\epsilon \to 0+$. We only show that $\epsilon^{-1} K(k, k') \to K^{(N)}_{\alpha, \Bessel} (x, y)$ for $(k, k') = \epsilon^{-1}(x, y) + 2 \epsilon^{-1} \log \epsilon^{-1}$. One then further shows, using a similar argument, that $\tr K$ on ${l^2 \{\ell+\tfrac12,\ell+\tfrac32,\dots\}}$ converges to $\tr K^{(N)}_{\alpha, \Bessel}$ on $L^2(s, \infty)$ to conclude. There are analytic justifications needed for interchanging limits and integrals, based on dominated convergence, that are implicit everywhere and which we omit. 

Let us then prove $\epsilon^{-1} K (k, k') \to K^{(N)}_{\alpha, \Bessel} (x, y)$ in the regime above-described. First we use the fact that the infinite-length $q$-Pochhammer symbol (the $q$-Gamma function) converges to the Gamma function:
\begin{equation}
    \log (q^c; q)_{\infty} = -\frac{\pi^2}{6} \epsilon^{-1} + \left( \frac{1}{2} - c \right) \log \epsilon + \frac{1}{2} \log(2 \pi) - \log \Gamma(c) + O(\epsilon)
\end{equation}
with $q = e^{-\epsilon}, \epsilon \to 0+, c \notin -\N$. To keep the discrete kernel $\epsilon^{-1} K(k, k')$ finite we change variables $(z, w) = (e^{\epsilon \zeta}, e^{\epsilon \omega})$ and notice the contours transform to vertical up-oriented lines close to the imaginary axis, the $\zeta$ on the right and the $\omega$ on the left. Finally in this limit we can use Lemma 5.5 of~\cite{BB19} to estimate $\kappa(z, w) \sim \frac{\pi}{\epsilon N \sin \frac{(\zeta - \omega)}{N}}, \epsilon \to 0+$. Putting this all together, $\epsilon K(k, k')$ converges to the following kernel:
\begin{equation}
    (x, y) \mapsto \frac{1}{(2 \pi i)^2} \int_{i \R + \eta}  d \zeta \int_{i \R - \eta}  d \omega \frac{f(\zeta) e^{-x \zeta}}{f(\omega)e^{-y \omega}} \frac{\pi}{N \sin \frac{(\zeta - \omega)}{N}}, \quad f(\zeta) := \frac{\Gamma(\frac{\alpha+1}{2} + \zeta)}{\Gamma(\frac{\alpha+1}{2} - \zeta)}
\end{equation}
(for some $\eta > 0$ small enough). To arrive at the stated form, we first notice $\frac{\pi}{ N \sin \frac{(\zeta - \omega)}{N}} = \int_{-\infty}^\infty \frac{e^{(N+\omega-\zeta)v} dv}{1+e^{Nv}}$. In the triple integral we observe Bessel functions appearing explicitly via~\cite[eq.~10.19.22]{DLMF} $J_\alpha(X) = \frac{1}{2 \pi i} \int_{-i \infty}^{i \infty} \frac{\Gamma(-Z) (X/2)^{\alpha + 2 Z} dZ}{\Gamma(\alpha+1+Z)}$
(with $Z = \zeta - \frac{\alpha+1}{2}, X = 2 \sqrt{e^{-x-v}}$, and similarly for $\omega$ and $y$). A final substitution $u=e^{-v}$ gives the stated form.

Finally, so far we obtained asymptotics for the shifted quantity $\lambda_1 + c - 1/2$, but the difference, after scaling, is $\epsilon (c-1/2)$ which converges to 0 in probability so the shift can be removed and part (i) follows.

We now turn to the proof of part (ii), and the only difference with the above is the asymptotic analysis. Let us first put $R = \epsilon^{-1} \to \infty$ and $b = \sqrt{a}$ for notational convenience. The strategy is the same, but now we will prove we have $R^{1/3} K(k, k') \to K^{(\beta c_2)}_\Airy (x, y)$ with $K$ the discrete kernel from above and with $(k, k') = c_1 R + (x, y) c_2 R^{1/3}$ where we recall $c_1 = -2 \log(1-b), c_2 = 2^{-1/3} b^{1/3} (1-b)^{-2/3}$.

As $0<b<1$ is fixed, we can use the estimate 
\begin{equation}
    (B q; q)_{\infty} \sim -\epsilon^{-1} \Li(B), \quad q = e^{-\epsilon}, \quad \epsilon \to 0+
\end{equation} 
(if $B$ is away from $0$ and $1$). Here $\Li$ is the dilogarithm function. Asymptotically then we estimate $R^{1/3} K(k, k')$ to be 
\begin{equation}
    \frac{R^{1/3}}{(2 \pi i)^2} \oint \oint \frac{e^{R (S(z) - S(w))}} {z^{k_0+1} w^{-k'_0+1}} \cdot \kappa(z, w)dzdw 
\end{equation}
where $S(z) = \Li(bz) - \Li(b/z) - c_1 \log z$ and $(k_0, k'_0) = (k, k') - c_1 R$. The main contribution of this integral then comes from the double critical point of the action $S$, and the argument as in~\cite[Prop.~5.1]{BB19} goes through with little modification. The double critical point of $S$ is at $z=1$, and to see a finite contribution of the action we rescale $z$ and $w$ zooming in around 1 as $(z, w) = (e^{\zeta R^{-1/3}/c_2}, e^{\omega R^{-1/3}/c_2})$. The action $S$ has a Taylor expansion around 1 given by 
\begin{equation}
    S(1) + \frac{S''' (1)} {6} \frac{\zeta^{3}}{c_2^3} - x \zeta + O(R^{-1/3}) = S(1) + \frac{\zeta^{3}}{3} - x \zeta + O(R^{-1/3})
\end{equation}
where $c_2$ was chosen explicitly to cancel $S'''(1)$. Furthermore, let us note that as $N = \beta R^{2/3}$ we have $q^N = e^{- \beta R^{-1/3}}$ (this is the ``important'' parameter in $\kappa(z, w)$) and so again we can estimate $\kappa(z, w) \sim \frac{\pi}{R^{1/3} \beta c_2 \sin \frac{\zeta - \omega}{\beta c_2}}$.

We then arrive, as $R \to \infty$, to the following limiting kernel for $R^{1/3} K(k, k')$:
\begin{equation}
    (x, y) \mapsto \frac{1}{(2 \pi i)^2} \int_{i \R + \eta} d \zeta \int_{i \R - \eta} d \omega \frac{e^{\zeta^3/3 - x \zeta}}{e^{\omega^3/3 - y \omega}} \frac{\pi} {\beta c_2 \sin \frac{\zeta - \omega}{\beta c_2}} 
\end{equation}
(for some $\eta$ small enough). We use the same integral trick as in part (i) above to transform $\frac{\pi}{\beta c_2 \sin \frac{\zeta - \omega}{\beta c_2}}$ and we recognize that $\Ai(x) = \frac{1}{2 \pi i}\int_{\eta+i \R} e^{\zeta^3/3 - x \zeta} d \zeta$. This finally leads to the stated form of $K^{(\beta c_2)}_\Airy (x, y)$ and modulo the same analytical arguments omitted above, we conclude the proof.

\section{Conclusion}

In this work, using a mixture of algebraic combinatorics, mathematical physics, and asymptotic analysis, we analyzed certain asymptotic distributions of peaks of simple cylindric plane partitions. Our work is far from exhaustive and builds on work of the authors~\cite{BO21} and of Betea--Bouttier~\cite{BB19}, and these works already build and expand upon work of Okounkov~\cite{Oko01}, Okounkov--Reshetikhin~\cite{OR03}, and Borodin~\cite{Bor07}. 

We identify four further directions worth investigating, besides what was already commented on in the remarks in the introduction. One is to consider cylindric plane partitions with more complex features (e.g.~different profiles, more than one peak, etc.), or indeed other types of tilings like domino tilings, and to study their ``edge'' asymptotic behavior. We plan to address this in future work. The second is to make our treatment of last passage percolation on a cylinder/reflecting strip systematic, and this is current joint work in progress of the first author with J. Bouttier, P. Nejjar, and M. Vuleti\'c. Thirdly, we would like to explore possible connections to the stochastic heat equation and the KPZ equation in more detail, via the bridge built in~\cite{IMS21} and the partition function of $\log-\Gamma$ polymers. Fourth, and as results of~\cite{BO21} show, it should be easy to obtain finite temperature analogues of other random matrix-type kernels like the Meijer-G kernel for example. We note however that we lose a direct connection to Muttalib--Borodin ensembles that was present in~\cite{BO21}.


\printbibliography

\end{document}